\documentclass[12pt]{article}
\usepackage{fullpage}
\usepackage{amssymb,amsmath,amsfonts,amsthm,latexsym,enumerate,url,cases}
\numberwithin{equation}{section}
\usepackage{hyperref}
\usepackage{indentfirst}
\hypersetup{colorlinks=true,citecolor=blue,linkcolor=blue,urlcolor=blue}

\newtheorem{theorem}{Theorem}[section] %
\newtheorem{lemma}[theorem]{Lemma} %

\newtheorem*{problem1}{Problem 1}%
\newtheorem*{problem2}{Problem 2}%
\newtheorem{remark1}{Remark}%

\newtheorem*{theoremA}{Theorem A}
\newtheorem*{theoremB}{Theorem B}
\newtheorem*{theoremC}{Theorem C}
\newtheorem*{theoremD}{Theorem D}
\newtheorem*{theoremE}{Theorem E}
\newtheorem*{theoremF}{Theorem F}
\newtheorem*{theoremG}{Theorem G}

\usepackage{cite}

\begin{document}
\title{On a problem of minimal additive complements for not eventually periodic $S$-difference sets}
\author{
 Min Tang\footnote{Corresponding author. This work is supported by the National Natural Science Foundation of China (Grant No. 12371003). E-mail: tmzzz2000@163.com (M. Tang), hwj021729@163.com (W.J. He).} and Wenjing He\\
\small  School of Mathematics and Statistics \\
\small  Anhui Normal University,  Wuhu  241002,  P. R. China\\
}
\date{}
\maketitle \baselineskip 14pt

{\bf Abstract.} Let $C$ and $W$ be two integer sets. If $C+W=\mathbb{Z}$, then we say that $C$ is an additive complement to $W$. If no proper subset of $C$ is an additive complement to $W$, then we say that $C$ is a minimal additive complement to $W$. In this paper, we give an affirmative answer to one problem of Ma and Chen [On a problem of minimal additive complements of integers, J. Number Theory 284(2026), 178-187.]



\section{Introduction}

 For $C,W\subseteq \mathbb{Z}$ and integer $a$, we define
$$C+W= \{c+w:c\in C,~w\in W\},$$ $$a+W=\{a+w:w\in W\},\quad a W=\{aw:w\in W\}.$$
For integers $a$ and $b$, $[a, b]$ is defined as the set of all integers between $a$ and $b$.
For a positive integer $m$, let $[m]$ denote the set $\{0, 1, \ldots, m-1\}$.
For any set $A \subseteq \mathbb{Z}$, let
$$[A]_m = \{r : r \in [m], r \equiv a \pmod{m}, a \in A\}.$$
Let $V = \{v_1 < v_2 < \cdots\}$ be an infinite set. Suppose that $[V]_m = \{r_1, \dots, r_s\}$. For $1 \leqslant i \leqslant s$, let
$$V_i = \{v \in V : v \equiv r_i \pmod m\} = \{v_{i_1} < v_{i_2} < \cdots\}.$$
We call set $V$ has {\it the big-hole property} modulo $m$ if for every $1 \leqslant i \leqslant s$, $V_i$
is an infinite set and $$\limsup_{k \to \infty} (v_{i_k} - v_{i_k-1}) = +\infty, \quad \limsup_{k \to \infty} (v_{i_k+1} - v_{i_k}) = +\infty,$$
where $v_{i_k} \in V_i$ and $v_{i_k-1}, v_{i_k+1}$ may not be in $V_i$.

If there exists a positive integer $T$ such that $x+T\in X$ for all sufficiently large integers $x\in X$, then we call $X$ eventually periodic with period $T$.
The set $C$ is called an additive complement to $W$ if $C+W=\mathbb{Z}$ and if no proper subset of $C$ is an additive complement to $W$, then the set $C$ is called a minimal additive complement to $W$.

In 2011, Nathanson \cite{N2011} initiated the study of minimal complements in infinite abelian groups. He proved that every complement to $W$ contains a minimal complement when $W$ is a finite set of integers. For infinite integers set $W$, he posed the following problem:

\begin{problem1}\label{prob1}(\cite[Problem 11]{N2011}). Let $W$ be an infinite set of integers. Does there exist a minimal complement to $W$? Does there exist a complement to $W$ that does not contain a minimal complement?
\end{problem1}

 In 2012, Chen and Yang \cite{CY2012} were the first to study the existence of minimal complements for infinite sets $W$.
\begin{theoremA}(See \cite[Theorem 1]{CY2012}.) Let $W$ be a set of integers with $\inf W=-\infty$ and $\sup W=+\infty$. Then there exists a minimal complement to $W$.
\end{theoremA}

\begin{theoremB} (See \cite[Theorem 2]{CY2012}.) Let $W=\{1=w_1<w_2<\cdots\}$ be a set of integers and $\overline{W}=(\mathbb{Z}\cap (0,+\infty))\setminus W=\{\overline{w_1}<\overline{w_2}<\cdots\}$.\\
(a) If $\lim\sup_{i\rightarrow+\infty}(w_{i+1}-w_i)=+\infty$, then there exists a minimal complement to $W$.\\
(b) If $\lim_{i\rightarrow+\infty}(\overline{w_{i+1}}-\overline{w_i})=+\infty$, then there does not exist a minimal complement to $W$.
\end{theoremB}

In 2019, Kiss, S\'{a}ndor and Yang \cite{KSY2019} investigated $W$ with
$\lim\sup_{i\rightarrow+\infty}(w_{i+1}-w_i)<+\infty$ which have a minimal complement.
They introduced the concept of eventually periodic sets
and established a modular criterion for determining the existence of minimal complements for this class of sets. Soon after, Kwon \cite{K2019} also made some new progress on this topic.

Let $W$ be a set of integers bounded below such that
\begin{equation}\label{eq1.1}
W = (m\mathbb{N} + X) \cup U \cup V,
\end{equation} where $X, U, V$ sets of integers with $[U]_m \subseteq [X]_m$ and $[X]_m, [V]_m$ are disjoint such that $V$ does not contain an arithmetic progression of common difference $m$. Let $W$ be as in (\ref{eq1.1}). If there exists a positive integer $m$ such that $X$ is non-empty and $V$ is finite, then $W$ is an eventually periodic set.
 If $X\neq \emptyset$ and there exists a set $\mathcal{C} \subseteq [m]$ such that $[\mathcal{C} + (X \cup V)]_m = [m]$, and for every $c \in \mathcal{C}$, there exists $v \in V$ such that $c+v
\not\equiv c'+x \pmod m$ for all $c' \in \mathcal{C} \setminus \{c\}$ and $x \in X \cup V$, then we say $W$ satisfies the {\it $\mathcal{C}$-condition modulo $m$.}

Recently, Ma and Chen obtained the following result:

 \begin{theoremC} (See \cite[Theorem 1.1]{Ma26}.) Let $W$ be as in (\ref{eq1.1}). If $W$ satisfies the {\it $\mathcal{C}$-condition modulo $m$} and $V$
 has the {\it big-hole property modulo $m$}, then $W$ has a minimal complement.
 \end{theoremC}

 For not eventually periodic set with bounded gaps, the existence question is even less understood.
In 2019, Kiss, S\'{a}ndor and Yang \cite{KSY2019} constructed the first example: an infinite, not eventually periodic set $W\subseteq \mathbb{N}$ with $w_{i+1}-w_i\in\{1,2\}$ which has a minimal additive complement.
\begin{theoremD}(See \cite[Theorem 4]{KSY2019}.) There exists an infinite, not eventually periodic set $W\subseteq \mathbb{N}$ such that $w_{i+1}-w_i\in\{1,2\}$ for all $i$ and there exists a minimal complement to $W$.
\end{theoremD}

Recently, Ma and Chen \cite{Ma26} generalized Theorem~D by proving that if $T \geqslant 2$, then there exists a not eventually periodic set $W = \{w_1 < w_2 < \cdots\}$ of natural numbers such that $W$ has a minimal additive complement, $w_{i+1} - w_i \in \{1, 2, \ldots, T\}$ for all $i$, and for each $1 \leqslant t \leqslant T$, $w_{i+1} - w_i = t$ holds for infinitely many integers $i$. They posed the following problem for further research.

\begin{problem2}(See \cite[Problem 1.3]{Ma26}.)Is it true that for any finite set $S$ of positive integers with $|S|\geq 2 $, there exists an infinite, not eventually periodic set $W = \{ w_1 < w_2 < \cdots \} \subseteq \mathbb{N}$ such that $W$ has a minimal additive complement, $w_{i+1} - w_i \in S$ for all $i$, and for each $s \in S$, $w_{i+1} - w_i = s$ holds for infinitely many integers $i$?
\end{problem2}

In fact, the authors of the present paper, together with our collaborators have made useful attempts on special cases of Problem 2: The following three results together cover several important subcases for $l=2$. Theorem~E solves the case $s_1\mid s_2$; Theorem~F and Theorem~G handle the case $s_1\geqslant 2$ and $\gcd(s_1,s_2)=1$.

\begin{theoremE}(See \cite[Theorem 1.3]{chen2026}.) Let $c$ and $d$ be two different positive integers with $c\mid d$. Then there exists an infinite, not eventually periodic set $W=\{w_1,w_2,\ldots\}$ such that $w_{i+1}-w_{i}\in\{c,d\}$ for all positive integers $i$ and there exists a minimal complement to $W$.
\end{theoremE}

\begin{theoremF}(See \cite[Theorem 1.2]{Guo2025}.) Let $m$ and $h$ be positive integers with $m\geq2$, $h\equiv 1\pmod{m}$. There exists an infinite, not eventually periodic set $W=\{w_i\}_{i=1}^{\infty}\subseteq \mathbb{N}$ such that $w_{i+1}-w_i\in\{m,h\}$ for all $i$ and there exists a minimal complement to $W$.
\end{theoremF}

\begin{theoremG}(See \cite[Theorem 1.1]{He2026}.)\label{theoremG}Let $m\geq 3$ and $2\leq r\leq m-1$ be positive integers with $(r,m)=1$. There exists an infinite, not eventually periodic set $W\subseteq \mathbb{N}$ such that $w_{i+1}-w_i\in\{m,m+r\}$ for all $i$ and there exists a minimal complement to $W$.
\end{theoremG}

\begin{remark1}\label{rem1} Under the same hypotheses on $m,r$ as in Theorem G, the conclusion remains valid with $\{m,m+r\}$ replaced by $\{m,h\}$ for any $h\equiv r\pmod{m}$.
\end{remark1}

Let $S$ be a finite set of positive integers with $|S|\geqslant 2$ and $W = \{ w_1 < w_2 < w_3 < \dots \} \subseteq \mathbb{N}$
be an infinite increasing sequence. We call $W$ an \textit{$S$-difference set} if it satisfies: For every $i \geqslant 1$, we have $w_{i+1} - w_i \in S$ and
each $s \in S$ occurs as the difference $w_{i+1} - w_i$ for infinitely many indices $i$. If, in addition, $W$ is not eventually periodic, then we say $W$ is an infinite, not eventually periodic $S$-difference set, abbreviated as \textit{INEP $S$-difference set}.

Motivated by the constructive ideas developed in Theorem F and Theorem G, we unify the proof framework and give an affirmative answer to Problem~2 by extending the scope to arbitrary $S$-difference sets. Our main result is as follows:

  \begin{theorem}\label{thm1} Let $l\geqslant 2$ be a positive integer and $S=\{s_1, s_2, \dots, s_l\}$ be a set of positive integers with $1\leqslant s_1<s_2<\cdots<s_l$. There exists an {\it INEP $S$-difference set} $W$ and there exists a minimal complement to $W$.
\end{theorem}

The study of minimal additive complements has attracted widespread attention since Nathanson's work. For related problems see \cite{AL2021}-\cite{BS2021AAM}, \cite{BL2023}, \cite{K2019},\cite{Z2023}.

 \section{Lemmas}

 \begin{lemma}\label{lem1} (See \cite[Theorem 1]{Brauer42}.) Let $a_1 \leqslant a_2 \leqslant \dots \leqslant a_k$ be relatively prime positive integers. We put
$$S = S(a_1, a_2, \dots, a_k) = a_2 + a_3 + \dots + a_{k-1} + a_1 a_k.$$
For $n > S$ the Diophantine equation
$$a_1 x_1 + a_2 x_2 + \dots + a_k x_k = n$$
always has solutions in positive integers $x_1, x_2, \dots, x_k$.
\end{lemma}

 \begin{lemma}\label{lem2} Let $d,k\geqslant 2$ be positive integers and $S=\{ds_1,\ldots,ds_k\}$ be a set of positive integers with $\gcd(s_1,\ldots,s_k)=1$.
 If $W = \{ w_1 < w_2 < \cdots \} \subseteq \mathbb{N}$ is an $\{s_1,\ldots,s_k\}$-difference set and $W$ has a minimal additive complement, then $dW$ is an $S$-difference set and $dW$ has a minimal additive complement.
  \end{lemma}

\begin{proof} It is easy to see that if $W$ is an $\{s_1,\ldots,s_k\}$-difference set, then $d W$ is an $S$-difference set.

Suppose that $C$ is a minimal additive complement to $W$.
 Let
$$C_d=\bigcup_{r=0}^{d-1}(dC + r).$$
We shall show that $C_d$ is a minimal additive complement to $d W$.

First, we prove that $C_d$ is an additive complement to $d W$.
For any $n \in \mathbb{Z}$, write $n= dq+r$, where $0 \leq r < d$. Since
$C+W=\mathbb{Z}$, there exist $w \in W$, $c \in C$ such that $w+c=q$.
Thus
$$n=dq+r=dw+dc+r\in d W+C_d.$$

Next, we show that for any $x \in C_d$, $d W+(C_d\setminus\{x\})\neq \mathbb{Z}.$

Write $x=dc_0+r_0$ for some $c_0 \in C$ and $0 \leq r_0 < d$. Since
$W+(C\setminus\{c_0\})\neq \mathbb{Z}$, there exists an integer $q_0\notin W+(C\setminus\{c_0\})$.
Then $$dq_0+r_0\notin d W+(C_d\setminus\{x\}).$$
Otherwise, if there exist $w'\in W$, $c'\in C$ and $0\leqslant r'\leqslant d-1$ such that $$dq_0+r_0=dw'+dc'+r',$$
then $w'+c'=q_0$ and $r'=r_0.$
It follows that $c'\neq c_0$, which contradicts that $q_0\notin W+(C\setminus\{c_0\})$.

  This completes the proof of Lemma \ref{lem2}.
\end{proof}

\section{Proof of Theorem \ref{thm1}}
By Lemma~\ref{lem2}, it is sufficient to prove Theorem \ref{thm1} holds for$$\gcd(s_1,\ldots,s_l)=1.$$
We split the proof into two cases: $s_1=1$ and $s_1\geqslant 2$.
Although the proof strategies are parallel, the concrete constructions differ in many details,
therefore we present them independently for better readability.
Write $$s=s_2+\cdots+s_{l-1}+s_1s_l.$$
\subsection{Case 1: $s_1=1$.}
 We construct sequences $\{d_i\}_{i=1}^\infty$,
$\{W_i\}_{i=1}^\infty$ and $\{c_i\}_{i=1}^\infty$ inductively.

 Put
$$d_1 = -1,\qquad c_1 = -2s-2,\qquad W_1 = \{0,1,\dots,4s+4\}.$$
 For $i=2,3,\ldots$, let $C_{i-1}=\{c_1,\dots,c_{i-1}\}$ and define
$$d_i = \max\{\,n<0 : n\notin W_{i-1}+C_{i-1}\,\},$$
choose $c_i$ such that
\begin{equation}\label{eq3.1}c_i\leqslant d_i + 2c_{i-1} - 3s.\end{equation}

For each $i\ge 2$, we construct three disjoint sets
$V_{i-1}^{(1)}, V_{i-1}^{(2)}, V_{i-1}^{(3)}$ such that their union
with $W_{i-1}$ forms $W_i$.

Step 1. Construct the elements of $W$ which lie in $(-2c_{i-1},\, d_i-c_i]$. Write $$x_i:=d_i-c_i+2c_{i-1}-s_l.$$
By (\ref{eq3.1}) we have
$$x_i>d_i-c_i+2c_{i-1}-s\geqslant 2s.$$

Choose
\begin{equation}\label{eq3.2}V_{i-1}^{(1)}=\left\{-2c_{i-1}+s_l\right\}
\cup\left\{ -2c_{i-1}+s_l+t:t=1,\ldots,x_{i}\right\}
\end{equation}
Then $\max V_{i-1}^{(1)} = d_i-c_i$.

By the construction of $V_{i-1}^{(1)}$, we have the following fact:

{\bf FACT I}. For all $i\geqslant 1$, $d_{i+1}-d_i \leqslant -2s-2$.

 Since $c_1+W_1 = [-2s-2,\,2s+2]$, we have $d_2 = -2s-3$. Thus $d_2-d_1=-2s-2$. For each $i=2,\ldots,$
by the construction of $V_{i-1}^{(1)}$ in Step 1, we have
\begin{equation}\label{Prop-1}
d_i-c_i-n\in V_{i-1}^{(1)}\qquad\text{for }n=0,\dots,2s+1,\tag{Prop-1}\end{equation}
hence,
$$\{d_i, d_i-1,\ldots,d_i-2s-1\}\subseteq c_i+V_{i-1}^{(1)}.$$
So, $d_{i+1}-d_i\leqslant -2s-2$.

Step 2. Construct the elements of $W$ which lie in $(d_i-c_i,\, d_1-c_i-1]$.

We first construct the elements of $V_{i-1,0}^{(2)}$ which fall in each $(d_i-c_i, d_{i-1}-c_i-1]$.

By FACT I, we have
$$d_{i-1}-c_{i}-1-\left(d_{i}-c_{i}\right)=d_{i-1}-d_{i}-1\geqslant 2s+1,$$
combining with Lemma \ref{lem1}, there exist positive integers $x_{i,j}(j=1,2,\ldots,l)$ such that
$$d_{i-1}-c_{i}-1-\left(d_{i}-c_{i}\right)=s_1x_{i,1}+s_2x_{i,2}+\cdots+s_lx_{i,l}.$$
Choose \begin{eqnarray}\label{eq3.3}V_{i-1,0}^{(2)}&=&\left\{d_{i}-c_{i}+\lambda_{1,0}: \lambda_{i,1}=1,\ldots,x_{i,1}\right\}
\\&\cup&\left\{d_{i}-c_{i}+x_{i,1}+s_2\lambda_{i,2}: \lambda_{i,2}=1,\ldots,x_{i,2}\right\}\cup\cdots\nonumber
\\&\cup&\left\{d_{i}-c_{i}+x_{i,1}+s_2x_{i,2}+\cdots+s_l\lambda_{i,l}: \lambda_{i,l}=1,\ldots,x_{i,l}\right\}.\nonumber\end{eqnarray}

For $i=3,\ldots$, noting that $$(d_{i-1}-c_{i}-1, d_{1}-c_{i}-1]=\bigcup\limits_{u=1}^{i-2}(d_{i-u}-c_{i}-1, d_{i-u-1}-c_{i}-1],$$
we will proceed to choose all the elements of $V_{i-1,u}^{(2)}$ which fall in $(d_{i-u}-c_{i}-1, d_{i-u-1}-c_{i}-1]$ for all $u=1,\ldots, i-2$.

By FACT I, we have
$$d_{i-u-1}-c_{i}-1-\left(d_{i-u}-c_{i}-1\right)=d_{i-u-1}-d_{i-u}>s,$$
combining with Lemma \ref{lem1}, there exist positive integers $y_{i,j}^{(u)}(j=1,2,\ldots,l)$ such that
$$d_{i-u-1}-c_{i}-1-\left(d_{i-u}-c_{i}-1\right)=s_1y_{i,1}^{(u)}+s_2y_{i,2}^{(u)}+\cdots+s_ly_{i,l}^{(u)}.$$
For each $u=1,\ldots,i-2$, choose \begin{eqnarray}\label{eeq3.4}V_{i-1,u}^{(2)}&=&\left\{d_{i-u}-c_{i}-1+s_l\lambda_{i,l}^{(u)}: \lambda_{i,l}^{(u)}=1,\ldots,y_{i,l}^{(u)}\right\}
\\&\cup&\left\{d_{i-u}-c_{i}+1+s_l y_{i,l}^{(u)}+s_1\lambda_{i,1}^{(u)}: \lambda_{i,1}^{(u)}=1,\ldots,y_{i,1}^{(u)}\right\}\cup\cdots\nonumber
\\&\cup&\left\{d_{i-u}-c_{i}+1+s_l y_{i,l}^{(u)}+s_1 y_{i,1}^{(u)}+\cdots+s_{l-1} \lambda_{i,l-1}^{(u)}: \lambda_{i,l-1}^{(u)}=1,\ldots,y_{i,l-1}^{(u)}\right\}.\nonumber\end{eqnarray}
By (\ref{eeq3.4}), we know that the smallest element in $V_{i-1,u}^{(2)}$ is $d_{i-u}-c_{i}+s_l-1$. Thus, for all $u=1,\ldots, i-2$, we have
\begin{equation}\label{Prop-2} d_{i-u}-c_{i}\notin V_{i-1,u}^{(2)}.\tag{Prop-2}\end{equation}

Write $$V_{i-1}^{(2)}=\bigcup\limits_{u=0}^{i-2}V_{i-1,u}^{(2)}.$$
By (\ref{eq3.3}) and (\ref{eeq3.4}), we have $$\max V_{i-1}^{(2)}=d_1-c_i-1, \quad i=2,3,\ldots$$

Step 3. Construct the elements of $V_{i-1}^{(3)}$ which fall in each $\left(d_1-c_i-1, -2c_{i}\right]$.

Since
$$-c_{i-1}-c_{i}-(d_1-c_{i}-1)=-c_{i-1}+2\geqslant -c_1+2>s,$$
combining with Lemma \ref{lem1}, there exist positive integers $z_{i,j}(j=1,2,\ldots,l)$ such that
$$-c_{i-1}-c_{i}-(d_1-c_{i}-1)=s_1z_{i,1}+s_2z_{i,2}+\cdots+s_lz_{i,l}.$$
Choose \begin{eqnarray}\label{eq3.5}V_{i-1}^{(3)}&=&\left\{d_1-c_i-1+s_l\beta_{i,l}:\beta_{i,l}=1,\ldots,z_{i,l}\right\}
\\&\cup&\left\{d_1-c_i-1+s_l z_{i,l}+s_{l-1}\beta_{i,l-1}:\beta_{i,l-1}=1,\ldots,z_{i,l-1}\right\}\cup \cdots\nonumber
\\ &\cup&\left\{d_1-c_i-1+s_l z_{i,l}+s_{l-1}z_{i,l-1}+\cdots+s_{2}\beta_{i,2}: \beta_{i,2}=1,\ldots,z_{i,2}\right\}\nonumber
\\&\cup&\left\{-2c_{i}-\beta: \beta=0,\ldots,-c_{i}+c_{i-1}+z_{i,1}\right\}.\nonumber\end{eqnarray}
By (\ref{eq3.5}), the smallest element in $V_{i-1}^{(3)}$ is $d_{1}-c_{i}+s_l-1$, thus
\begin{equation}\label{Prop-3} d_{1}-c_{i}\notin V_{i-1}^{(3)}.\tag{Prop-3}\end{equation}
Moreover, $$\max V_{i-1}^{(3)}=-2c_{i}, \quad i=2,3,\ldots.$$

 For $i=2,3,\dots$, let
$$W_i = W_{i-1}\cup V_{i-1}^{(1)}\cup V_{i-1}^{(2)}\cup V_{i-1}^{(3)}.$$

Write $$W = \bigcup_{i=1}^\infty W_i,\quad C = \bigcup_{i=1}^\infty \{c_i\}.$$

Now, we summarize the construction of $W$ as follows:

 \textbf{Initial stage}: $W_1 = \{0, 1, \dots, 4s+4\}$. All consecutive differences are $s_1\in S$.

 \textbf{For each $i \ge 2$}, three pairwise disjoint sets are added to $W_{i-1}$:
   \begin{itemize}
    \item $V_{i-1}^{(1)}$: Starting from $-2c_{i-1} + s_l$, add consecutive integers with step $1$ up to $d_i - c_i$. All gaps here are $1$.

        \item  For each subinterval, use Lemma~\ref{lem1} to express its length as a positive linear combination of $s_1, s_2, \dots, s_l$. Then fill the subinterval by taking steps equal to these values in the given order: for the first subinterval ($u=0$) the order is $s_1, s_2, \dots, s_l$; for later subintervals ($u=1,\ldots,i-2$) the order is $s_l, s_1, s_2, \dots, s_{l-1}$. All internal gaps belong to $S$.

        \item $V_{i-1}^{(3)}$: Again use Lemma \ref{lem1}, start with a block of step $s_l$, then step $s_{l-1}$, \dots, step $s_2$, and finally step $s_1$ down to $-2c_i$. All gaps are in $S$.
   \end{itemize}

   \textbf{Transitions}: The maximum of $V_{i-1}^{(3)}$ is $-2c_i$ and the minimum of the next $V_i^{(1)}$ is $-2c_i + s_l$; the gap equals $s_l \in S$. Other connections between blocks are similarly arranged so that every consecutive difference of $W$ lies in $S$.

In all, we know that $W$ is an \textit{$S$-difference set}. Furthermore, by (\ref{eq3.1}), we can choose suitable parameters
$c_{i}$ such that $W$ is not eventually periodic. Hence, $W$ is an \textit{INEP $S$-difference set}.

Next, we shall show that $C$ is a minimal additive complement to $W$.

Firstly, $C$ is an additive complement to $W$.
Since $d_{1}=-1$ and by FACT I, we have $d_{i+1}-d_{i}\leqslant -2s-2$ for all $i\geq 1$, thus $d_k\rightarrow-\infty$. So, we have $(-\infty,2s+2]\subseteq W+C$. For any integer $n\geqslant 2s+2$, there exists an
$i\geqslant 2$ such that $-c_{i-1}\leqslant n<-c_i$.
Hence $$-c_{i}+d_{1}<-c_{i-1}-c_{i}\leqslant n-c_{i}<-2c_{i}.$$
The difference between adjacent elements of the interval $V_{i-1}^{(3)}\cap\left[-c_{i-1}-c_{i},-2c_{i}\right]$ is 1.
Thus, $n-c_{i}\in V_{i-1}^{(3)}$. Hence, $$n = c_i+(n-c_i)\in c_i+W_i.$$ So, $W+C=\mathbb{Z}$.

Secondly, $C$ is a minimal additive complement to $W$. For any positive integer $i$, write $d_{i}=c+w$ with $c\in C$ and $w\in W$. We shall prove that $c=c_i$.

We have $d_{1}-c_{1}=2s+1\in W_{1}.$ By (Prop-1), we have \begin{equation}\label{eeq2.12}d_{i}-c_{i}\in W\text{ for all } i\geqslant 2.\end{equation}

Fix an $i\geqslant 2$, by (Prop-2) we know that for any $j\geqslant i+1$ we have
$d_{i}-c_{j}\notin V_{j-1}^{(2)}.$
By (Prop-3), for all $j\geqslant 2$, we have $d_{1}-c_{j}\notin V_{j-1}^{(3)}.$
Hence, for any fixed $i\geqslant 1$, if $j\geqslant i+1$, then \begin{equation}\label{eeq2.13}d_{i}-c_{j}\notin W.\end{equation}
That is, if $d_{i}=c+w$, then $c\neq c_{j}$ for all positive integers $j>i$.

By the construction of $V_{i-1}^{(1)}$ in Step 1, we have
$$\min (W\backslash W_{i-1})= -2c_{i-1}+s_l,$$
thus for all $j\leqslant i-1$, we have $$d_{i}-c_{j}\leqslant d_{i}-c_{i-1}<-2c_{i-1},$$
it follows that for all positive integer $j\leqslant i-1$, we have $$d_{i}-c_{j}\notin W\backslash W_{i-1}.$$
By the definition of $d_i$, we know that $$d_{i}\notin W_{i-1}+\left\{c_{1},\ldots, c_{i-1}\right\}.$$
Hence, \begin{equation}\label{eeq2.14}d_{i}\notin W+\left\{c_{1}, \ldots, c_{i-1}\right\}.\end{equation}

So, by (\ref{eeq2.12})-(\ref{eeq2.14}), we know that if $d_{i}=c+w$, then $c=c_{i}$.

Therefore, $C$ is a minimal complement to $W$.

This completes the proof of the case $s_1=1$.

\subsection{Case 2: $s_1\geqslant 2$.}

Since $\gcd(s_1,s_2,\ldots,s_l)=1$ and $l\geqslant 2$, there exists an integer $2\leqslant j\leqslant l$ such that $s_1\nmid s_j$. Write $$s_j=ks_1+r,\quad 1\leqslant r\leqslant s_1-1.$$
Define
\[h(s_1;r):=\left\{
\begin{array}{ll}s_1-1, &\text{ if } r=1,\\
\left\lfloor\displaystyle\frac{s_1}{r}\right\rfloor, &\text{ if } s_1\equiv 1\pmod r,\\
\left\lceil\displaystyle\frac{s_1}{r}\right\rceil,&\text{ otherwise}.
\end{array}
\right. \]

We construct sequences $\{d_{i}\}^{\infty}_{i=1}, \{W_{i}\}^{\infty}_{i=1}$ and $\left\{c_{i}^{(1)}, \ldots,c_{i}^{(s_1)}\right\}^{\infty}_{i=1}$ inductively.

Put $$d_{1}=-1,\; c_{1}^{(\delta)}=-s_1s-\delta,\; \delta=1,\ldots,s_1,\;W_{1}=\{s_1\theta: \theta=0,1,\ldots,2s+2\}.$$
For $i=2,3\ldots$, write $$C_{i-1}=\left\{c_{1}^{(1)} ,\ldots,c_{1}^{(s_1)}, \ldots, c_{i-1}^{(1)}, \ldots,c_{i-1}^{(s_1)}\right\}.$$

Define \begin{equation}\label{eq4.3}d_{i}=\max\left\{n<0: n\notin W_{i-1}+C_{i-1}\right\}-(r-1).\end{equation}
Choose $c_{i}^{(1)}$ satisfying \begin{equation}\label{eq4.4}c_{i}^{(1)}\equiv d_{i}+2c_{i-1}^{(s_1)} \pmod{s_1},\end{equation}
and
\begin{equation}\label{eq4.5}c_{i}^{(1)}\leqslant\min\left\{d_{i}+2c_{i-1}^{(s_1)}-2s_1 s, \;3c_{i-1}^{(1)}-2(s_1-1)-s_1s_j\right\}.\end{equation}
For $\delta=2,\dots,s_1$, define $c_{i}^{(\delta)}=c_{i}^{(1)}-\delta+1.$

For each $i=2,3,\ldots$, we construct three disjoint sets
$V_{i-1}^{(1)}, V_{i-1}^{(2)}, V_{i-1}^{(3)}$ such that their union
with $W_{i-1}$ forms $W_i$.

Step 1. Construct the elements of $W$ which fall in each $\left(-2c_{i-1}^{(s_1)}, d_i-c_{i}^{(1)}+s_1\right]$.

By (\ref{eq4.4}) and (\ref{eq4.5}), we have $$d_{i}-c_{i}^{(1)}+2c_{i-1}^{(s_1)}-s_1s_j\equiv 0 \pmod{s_1}$$ and
$$d_{i}-c_{i}^{(1)}+2c_{i-1}^{(s_1)}-s_1s_j>s_1s,$$
thus there exists $x_{i}> s$ such that
$d_{i}-c_{i}^{(1)}+2c_{i-1}^{(s_1)}-s_1s_j=s_1x_{i}.$
Choose
\begin{eqnarray}\label{eeq2.1}V_{i-1}^{(1)}&=&\left\{-2c_{i-1}^{(s_1)}+s_j\lambda:\lambda=1,\ldots,s_1\right\}\\ \nonumber
&\cup&\left\{ -2c_{i-1}^{(s_1)}+s_1s_j+s_1t:t=1,\ldots,x_{i}\right\}
\cup\left\{d_i-c_{i}^{(1)}+s_1\right\}. \nonumber
\end{eqnarray}

Noting that the largest element and the second largest element of the set $V_{i-1}^{(1)}$ are $d_i-c_{i}^{(1)}+s_1$ and $d_i-c_{i}^{(1)}$, respectively, thus\begin{equation}\label{iv-b}d_{i}-c_{i}^{(\delta)}=d_{i}-c_{i}^{(1)}+\delta-1\not\in V_{i-1}^{(1)}, \;\delta=2,\ldots,s_1. \tag{iv-a}\end{equation}
Moreover, we have the following fact:

{\bf FACT II} For all $i\geqslant 1$, we have $d_{i+1}-d_{i}\leqslant -s_1s-s_1-r+1$.

Since $$W_{1}+\left\{c_{1}^{(1)},\ldots,c_{1}^{(s_1)}\right\}=[-s_1s-s_1,s_1s+2s_1-1],$$
we have $d_{2}=-s_1s-s_1-r$, thus $d_{2}-d_{1}=-s_1s-s_1-r+1$.

For all $i\geqslant 2$, we have \begin{equation}\label{iv-a}d_{i}-c_{i}^{(1)}-s_1 n\in V_{i-1}^{(1)}, \quad n=0,\ldots,s.\tag{iv-b}\end{equation}
Noting that
\begin{eqnarray*}
d_{i}-s_1 n&=&(d_{i}-c_{i}^{(1)}-s_1n)+c_{i}^{(1)},\nonumber\\
d_{i}-(s_1 n+1)&=&(d_{i}-c_{i}^{(1)}-s_1 n)+c_{i}^{(2)},\\
\ldots\nonumber\\
d_{i}-(s_1n+s_1-1)&=&(d_{i}-c_{i}^{(1)}-s_1 n)+c_{i}^{(s_1)},\nonumber
\end{eqnarray*}
thus, for each $n=0,\ldots,s$, we have
\begin{equation}\label{eq4.6}\{d_{i}-s_1n,\dots,d_{i}-(s_1n+s_1-1)\}\subseteq V_{i-1}^{(1)}+\left\{c_{i}^{(1)},\dots,c_{i}^{(s_1)}\right\}.\end{equation}
Moreover, if $r\geqslant 2$, then
$$d_{i}+m=(d_{i}-c_{i}^{(1)}+s_1)+c_{i}^{(s_1+1-m)}, \;m=1,\ldots,r-1.$$
That is,
\begin{equation}\label{eq4.7}\{d_{i}+1,\dots,d_{i}+r-1\}\subseteq V_{i-1}^{(1)}+\left\{c_{i}^{(1)},\dots,c_{i}^{(s_1)}\right\},\end{equation}
By (\ref{eq4.6}) and (\ref{eq4.7}), we have $$d_{i+1}-d_{i}\leqslant -s_1s-s_1-r+1.$$

Hence, for all $i\geqslant 1$, we have $d_{i+1}-d_{i}\leqslant -s_1s-s_1-r+1$.

Step 2. Construct all the elements of $W$ which fall in $\left(d_i-c_{i}^{(1)}+s_1, d_1-c_{i}^{(1)}-1\right]$.

 For each $i=2,3,\ldots$, we first construct the elements of $W$ which fall in each $\left(d_i-c_{i}^{(1)}+s_1, d_{i-1}-c_{i}^{(1)}-1\right]$.

By FACT II and the definition of $h(s_1;r)$, it is easy to verify that $$d_{i-1}-c_{i}^{(1)}-1-s_jh(s_1;r)-\left(d_{i}-c_{i}^{(1)}\right)>s.$$
By Lemma \ref{lem1}, there exist positive integers $x_{i,j}(j=1,2,\ldots,l)$ such that
$$d_{i-1}-c_{i}^{(1)}-1-s_jh(s_1;r)-\left(d_{i}-c_{i}^{(1)}\right)=s_1x_{i,1}+s_2x_{i,2}+\cdots+s_lx_{i,l}.$$
Choose \begin{eqnarray}\label{eeq2.4}V_{i-1,0}^{(2)}&=&\left\{d_{i}-c_{i}^{(1)}+s_1\lambda_{i,1}: \lambda_{i,1}=1,\ldots,x_{i,1}\right\}\nonumber
\\&\cup&\left\{d_{i}-c_{i}^{(1)}+s_1x_{i,1}+s_2\lambda_{i,2}: \lambda_{i,2}=1,\ldots,x_{i,2}\right\}\cup\cdots
\\&\cup& \left\{d_{i}-c_{i}^{(1)}+s_1x_{i,1}+s_2x_{i,2}+\cdots+s_l\lambda_{i,l}: \lambda_{i,l}=1,\ldots,x_{i,l}\right\}\nonumber\\
&\cup& \left\{ d_{i-1}-c_{i}^{(1)}-1-s_j\lambda:\lambda=0,\dots,h(s_1;r)\right\}.\nonumber\end{eqnarray}

For $i\geq 3$, noting that
$$\left[d_{i-1}-c_{i}^{(1)}, d_{1}-c_{i}^{(1)}-1\right]=\bigcup\limits_{u=1}^{i-2}\left[d_{i-u}-c_{i}^{(1)}, d_{i-u-1}-c_{i}^{(1)}-1\right],$$
we will proceed to choose all the elements of $W$ which fall in $\left[d_{i-u}-c_{i}^{(1)}, d_{i-u-1}-c_{i}^{(1)}-1\right]$ for all $u=1,\ldots, i-2$.

Again by FACT II and the definition of $h(s_1;r)$, we have
\begin{eqnarray*}
d_{i-u-1}-c_{i}^{(1)}-1-s_j h(s_1;r)-\left(d_{i-u}-c_{i}^{(1)}-1\right)>s.\end{eqnarray*}
By Lemma \ref{lem1}, there exist positive integers $y_{i,j}^{(u)}(j=1,2,\ldots,l)$ such that
\begin{eqnarray*}
&&d_{i-u-1}-c_{i}^{(1)}-1-s_j h(s_1;r)-(d_{i-u}-c_{i}^{(1)}-1)\\
&&=s_1y_{i,1}^{(u)}+s_2y_{i,2}^{(u)}+\cdots+s_l y_{i,l}^{(u)}.\end{eqnarray*}

Thus,
\begin{eqnarray}\label{eeq2.5}&&V_{i-1,u}^{(2)}=\left\{d_{i-u}-c_{i}^{(1)}-1+s_j \lambda_{i,j}^{(u)}:\lambda_{i,j}^{(u)}=1,\ldots,y_{i,j}^{(u)}\right\}\\ \nonumber
&\cup&\left\{ d_{i-u}-c_{i}^{(1)}-1+s_j y_{i,j}^{(u)}+s_1 \lambda_{i,1}^{(u)}:\lambda_{i,1}^{(u)}=1,\ldots,y_{i,1}^{(u)}\right\}\cup \cdots\\ \nonumber
&\cup&\left\{d_{1}-c_{i}^{(1)}-1+s_j y_{i,j}^{(u)}+s_1y_{i,1}^{(u)}+\cdots+s_{j-1}\lambda_{i,j-1}^{(u)}:\lambda_{i,j-1}^{(u)}=1,\ldots,y_{i,j-1}^{(u)}\right\}\cup \cdots\\ \nonumber
&\cup&\left\{d_{1}-c_{i}^{(1)}-1+s_j y_{i,j}^{(u)}+\cdots+s_{j-1}y_{i,j-1}^{(u)}+s_{j+1}y_{i,j+1}^{(u)}+\cdots+s_{l}\lambda_{i,l}^{(u)}:\lambda_{i,l}^{(u)}=1,\ldots,y_{i,l}^{(u)}\right\}\\ \nonumber
&\cup&\left\{ d_{i-u-1}-c_{i}^{(1)}-1-s_j \theta:\theta=0,\dots,h(s_1;r)\right\}.\nonumber\end{eqnarray}
By (\ref{eeq2.5}), we know that the smallest element in $V_{i-1,u}^{(2)}$ is $d_{i-u}-c_{i}^{(1)}+s_j-1$, thus, we have
\begin{equation}\label{iv-c} d_{i-u}-c_{i}^{(\delta)}\notin V_{i-1,u}^{(2)}, \quad \delta=1,\ldots,s_1.\tag{iv-c}\end{equation}

Write $$V_{i-1}^{(2)}=\bigcup\limits_{u=0}^{i-2}V_{i-1,u}^{(2)}.$$
Then $$\max V_{i-1}^{(2)}=d_1-c_{i}^{(1)}-1, \quad i=2,3,\ldots$$

Step 3. Construct the elements of $W$ which fall in each $\left(d_{1}-c_{i}^{(1)}-1, -2c_{i}^{(s_1)}\right]$.

Let $a_{i}$ be the least nonnegative integer of $-2c_{i}^{(s_1)}+c_{i-1}^{(s_1)}+c_{i}^{(1)}$ modulo $s_1$. Noting that $$-c_{i-1}^{(s_1)}-c_{i}^{(1)}-(s_1-a_i)-\left(d_{1}-c_{i}^{(1)}-1\right)>s,$$
by Lemma \ref{lem1}, there exist positive integers $z_{i,j}(j=1,2,\ldots,l)$ such that
\begin{eqnarray*}
&&-c_{i-1}^{(s_1)}-c_{i}^{(1)}-(s_1-a_i)-\left(d_{1}-c_{i}^{(1)}-1\right)\\
&&=s_1z_{i,1}+s_2z_{i,2}+\cdots+s_l z_{i,l}.\end{eqnarray*}

Choose
\begin{eqnarray}\label{eeq2.6}&&V_{i-1}^{(3)}=\left\{d_{1}-c_{i}^{(1)}-1+s_j\lambda_{i,j}:\lambda_{i,j}=1,\ldots,z_{i,j}\right\}\\ \nonumber
&\cup&\left\{d_{1}-c_{i}^{(1)}-1+s_j z_{i,j}+s_1\lambda_{i,1}:\lambda_{i,1}=1,\ldots,z_{i,1}\right\}\cup \cdots\\ \nonumber
&\cup&\left\{d_{1}-c_{i}^{(1)}-1+s_j z_{i,j}+s_1z_{i,1}+\cdots+s_{j-1}\lambda_{i,j-1}:\lambda_{i,j-1}=1,\ldots,z_{i,j-1}\right\}\cup \cdots\\ \nonumber
&\cup&\left\{d_{1}-c_{i}^{(1)}-1+s_j z_{i,j}+\cdots+s_{j-1}z_{i,j-1}+s_{j+1}z_{i,j+1}+\cdots+s_{l}\lambda_{i,l}:\lambda_{i,l}=1,\ldots,z_{i,l}\right\}\\ \nonumber
&\cup&\left\{-2c_{i}^{(s_1)}-s_1 \beta:\beta=0,\dots,\frac{-2c_{i}^{(s_1)}+c_{i-1}^{(s_1)}+c_{i}^{(1)}+(s_1-a_i)}{s_1}\right\}.\nonumber\end{eqnarray}
Then $$\max V_{i-1}^{(3)}=-2c_{i}^{(s_1)}, \quad i=2,3,\ldots$$

By the construction of $V_{i-1}^{(3)}$ in (\ref{eeq2.6}), we know that the smallest element in $V_{i-1}^{(3)}$ is $d_{1}-c_{i}^{(1)}+s_j-1$, thus, we have
\begin{equation}\label{iv-d}d_{1}-c_{i}^{(\delta)}\notin V_{i-1}^{(3)}, \quad \delta=1,\ldots,s_1.\tag{iv-d}\end{equation}

For $i=2,3,\ldots$, let
\begin{equation}\label{eq520}W_{i}=W_{i-1}\cup V_{i-1}^{(1)}\cup V_{i-1}^{(2)}\cup V_{i-1}^{(3)},\end{equation}
$$C_{i-1}=\left\{c_{1}^{(1)} ,\ldots,c_{1}^{(s_1)}, \ldots, c_{i-1}^{(1)}, \ldots,c_{i-1}^{(s_1)}\right\}.$$
Write $$W=\bigcup\limits_{i=1}^{\infty}W_{i}, \quad C=\bigcup\limits_{i=1}^{\infty}\left\{c_{i}^{(1)},\dots,c_{i}^{(s_1)}\right\}.$$

Now, we summarize the construction of $W$ as follows:

\textbf{Initial stage:}
\(W_1 = \{ s_1\theta : \theta = 0,1,\dots,2s+2\}\). All consecutive differences are \(s_1\in S\).

\textbf{For each \(i \ge 2\)}, three pairwise disjoint sets are added to \(W_{i-1}\):
\begin{itemize}
    \item \(V_{i-1}^{(1)}\): Starting from \(-2c_{i-1}^{(s_1)}+s_j\), take \(s_1\) steps of size \(s_j\), then continue with steps of size \(s_1\) up to \(d_i-c_i^{(1)}\), and finally add the single element \(d_i-c_i^{(1)}+s_1\). All internal gaps belong to \(\{s_1,s_j\}\subseteq S\).

    \item \(V_{i-1}^{(2)}\): For each subinterval, use Lemma~\ref{lem1} to express its length as a positive linear combination of \(s_1,s_2,\dots,s_l\). Then fill the subinterval by taking steps equal to these values in the prescribed order.
    All internal gaps belong to \(S\).

    \item \(V_{i-1}^{(3)}\): Again use Lemma~\ref{lem1}, start with a block of step \(s_j\), then step \(s_1\), then step \(s_2\), \dots, step \(s_{l}\), and finally again with step \(s_1\) to \(-2c_i^{(s_1)}\).
\end{itemize}

\textbf{Transitions:} The maximum of \(V_{i-1}^{(3)}\) is \(-2c_i^{(s_1)}\) and the minimum of the next \(V_i^{(1)}\) is \(-2c_i^{(s_1)}+s_j\); the gap equals \(s_j\in S\). Other connections between blocks are similarly arranged so that every consecutive difference of \(W\) lies in \(S\).

In all, \(W\) is an \textit{\(S\)-difference set}. By choosing suitable parameters \(c_i^{(1)}\) we can ensure that \(W\) is not eventually periodic. Hence \(W\) is an \textit{INEP \(S\)-difference set}.

Next, we shall show that $C$ is a minimal additive complement to $W$.

Since $d_{1}=-1$ and $d_{i+1}-d_{i}\leqslant -s_1s-s_1-r+1$ for all $i\geq 1$, thus $d_k\rightarrow-\infty$. So, we have $(-\infty,s_1s+2s_1-1]\subseteq W+C$. For any integer
$n\geqslant s_1s+2s_1$, there exists an $i\geqslant 2$ such that $-c_{i-1}^{(s_1)}\leq n<-c_{i}^{(s_1)}$. Hence
$$-c_{i}^{(1)}+d_{1}<-c_{i-1}^{(s_1)}-c_{i}^{(1)}\leqslant n-c_{i}^{(1)}<\dots<n-c_{i}^{(s_1)}<-2c_{i}^{(s_1)}.$$
The difference between adjacent elements of the interval $\left[-c_{i-1}^{(s_1)}-c_{i}^{(1)},-2c_{i}^{(s_1)}\right]$ is $s_1$.
Thus, exactly one of ${n-c_{i}^{(1)},\dots,n-c_{i}^{(s_1)}}$ is in $V_{i-1}^{(3)}$. Hence, $$ n\in W_{i}+\left\{c_{i}^{(1)}, \dots,c_{i}^{(s_1)}\right\}.$$
So, $W+C=\mathbb{Z}$.

To prove $C$ is minimal, we first derive some properties from the construction of $W$.

Fix an $i\geqslant 2$, for $\delta=1,\dots,s_1$, let $\delta_r$ be the least nonnegative residue of $\delta$ modulo $r$. In particular, we write $\delta_r=1$ if $r=1$.

If $\delta_r\in \left\{0,1\right\}$, then by (iv-a) and (iv-b), we have
\begin{eqnarray*}
&&d_{i}-s_j\frac{\delta-\delta_r}{r}-(\delta_r-1)-c_{i}^{(\delta)}\\
&&=d_{i}-\left(ks_1\frac{\delta-\delta_r}{r}+\delta-1\right)-c_{i}^{(\delta)}\\
&&=d_{i}-c_{i}^{(1)}-ks_1\frac{\delta-\delta_r}{r}\in V_{i-1}^{(1)},
\end{eqnarray*}
and\begin{equation}\label{iv-e-1}d_{i}-s_j\frac{\delta-\delta_r}{r}-(\delta_r-1)-c_{i}^{(\nu)}\notin V_{i-1}^{(1)}, \quad \nu\neq \delta.\tag{iv-e-1}\end{equation}

By the construction of $V_{i-1}^{(2)}$ and $V_{i-1}^{(3)}$, we know that the distance of the adjacent element in $$ \left[d_{i-1}-c_{i}^{(1)}-1-s_jh(s_1;r), d_{i-1}-c_{i}^{(1)}+s_j-1\right]$$ is $s_j$,
thus, for any $0\leqslant \lambda \leqslant h(s_1;r)$, we have
\begin{align}&d_{i-1}-c_{i}^{(\delta)}-s_j\lambda-(\delta_r-1)\tag{iv-f-1}\\
=&d_{i-1}-c_{i}^{(1)}-1-s_j\lambda+\delta+1-\delta_r\notin V_{i-1}^{(2)}\cup V_{i-1}^{(3)},\nonumber
\end{align}
where $\delta=1,\ldots,s_1.$

If $i\geq 3$, then by the construction of (\ref{eeq2.5}) and (\ref{eeq2.6}), for all $u=1,\ldots, i-2$, we know that the distance of the adjacent element in$$\left[d_{i-u-1}-c_{i}^{(1)}-1-s_jh(s_1;r), d_{i-u-1}-c_{i}^{(1)}+s_j-1\right]$$
is $s_j$, thus, for any $0\leqslant \lambda \leqslant h(s_1;r)$, we also have
\begin{align}&d_{i-u-1}-c_{i}^{(\delta)}-s_j\lambda-(\delta_r-1)\tag{iv-g-1}\\
=&d_{i-u-1}-c_{i}^{(1)}-1-s_j\lambda+\delta+1-\delta_r\notin V_{i-1}^{(2)}\cup V_{i-1}^{(3)},\nonumber
\end{align}
where $\delta=1,\ldots,s_1.$

If $\delta_r\in \left\{2,\dots,r-1\right\}$, then by (iv-a) and (iv-b), we have
\begin{eqnarray*}
&&d_{i}-s_j\frac{\delta-\delta_r+r}{r}+(r+1-\delta_r)-c_{i}^{(\delta)}\\
&&=d_{i}-\left(ks_1\frac{\delta-\delta_r+r}{r}+\delta-1\right)-c_{i}^{(\delta)}\\
&&=d_{i}-c_{i}^{(1)}-ks_1\frac{\delta-\delta_r+r}{r}\in V_{i-1}^{(1)},
\end{eqnarray*}
and
\begin{equation}\label{iv-e-2}d_{i}-s_j\frac{\delta-\delta_r+r}{r}+(r+1-\delta_r)-c_{i}^{(\nu)}\notin V_{i-1}^{(1)}, \quad \nu\neq \delta.\tag{iv-e-2}\end{equation}

By the construction of $V_{i-1,0}^{(2)}$ in (\ref{eeq2.4}), we know that the distance of the adjacent element in $$V_{i-1,0}^{(2)}\cap\left[d_{i-1}-c_{i}^{(1)}-1-s_jh(s_1;r), d_{i-1}-c_{i}^{(1)}-1\right]$$ is $s_j$, thus for any $1\leqslant \lambda \leqslant h(s_1;r)$, we have
\begin{align}&d_{i-1}-c_{i}^{(\delta)}-s_j\lambda+(r+1-\delta_r)\tag{iv-f-2}\\
=&d_{i-1}-c_{i}^{(1)}-1-s_j\lambda+\delta+r+1-\delta_r\notin V_{i-1,0}^{(2)},\nonumber
\end{align}
where $\delta=1,\ldots,s_1.$

If $i\geq 3$, then by (\ref{eeq2.5}), for all $u=1,\ldots, i-2$, we know that the distance of the adjacent element in$$V_{i-1,u}^{(2)}\cap\left[d_{i-u-1}-c_{i}^{(1)}-1-s_jh(s_1;r), d_{i-u-1}-c_{i}^{(1)}-1\right]$$
is $s_j$, thus for any $1\leqslant \lambda\leqslant h(s_1,r)$, we have
\begin{align}&d_{i-u-1}-c_{i}^{(\delta)}-s_j\lambda+(r+1-\delta_r)\tag{iv-g-2}\\
=&d_{i-u-1}-c_{i}^{(1)}-1-s_j\lambda+\delta+r+1-\delta_r\notin V_{i-1,u}^{(2)},\nonumber
\end{align}
where $\delta=1,\ldots,s_1.$

Next, we shall prove the following important properties:

For any $1\leqslant\delta\leqslant s_1$ and $i\geqslant 2$, we have
\begin{equation}\label{iv-h-1}d_{i}-s_j\frac{\delta-\delta_r}{r}-(\delta_r-1)\notin W_{i-1}+C_{i-1}, \text{ if }\delta_r\in \{0,1\}.\tag{iv-h-1}\end{equation}
\begin{equation}\label{iv-h-2}d_{i}-s_j\frac{\delta-\delta_r+r}{r}+(r+1-\delta_r)\notin W_{i-1}+C_{i-1}, \text{ if }\delta_r\in \{2,\ldots,r-1\}.\tag{iv-h-2}\end{equation}

 Indeed, we have $$W_{1} + C_{1} = \left[-s_1s-s_1,\; s_1s+2s_1-1\right], \quad d_{2}=-s_1s-s_1-r.$$
Thus,
$$d_{2} -s_j\frac{\delta-\delta_r}{r}-(\delta_r-1) \notin W_{1} + C_{1}, \text{ if }\delta_r\in \{0,1\}.$$
$$d_{2}-s_j\frac{\delta-\delta_r+r}{r}+(r+1-\delta_r) \notin W_{1} + C_{1}, \text{ if }\delta_r\in \{2,\ldots,r-1\}.$$

For $i=2,3,\ldots$, by (\ref{eq520}) we have
\begin{eqnarray}\label{eq5201} &&W_i+C_i=(W_{i-1}\cup V_{i-1})+(C_{i-1}\cup \{c_{i}^{(1)}, \ldots, c_{i}^{(s_1)}\})\nonumber\\
&=&(W_{i-1}+C_{i-1})\cup (V_{i-1}+C_{i-1})\cup (W_{i-1}+\{c_{i}^{(1)}, \ldots, c_{i}^{(s_1)}\})\\
&\cup &(V_{i-1}^{(1)}+\{c_{i}^{(1)}, \ldots, c_{i}^{(s_1)}\})\cup (V_{i-1}^{(2)}+\{c_{i}^{(1)}, \ldots, c_{i}^{(s_1)}\})\cup (V_{i-1}^{(3)}+\{c_{i}^{(1)}, \ldots, c_{i}^{(s_1)}\}),\nonumber
\end{eqnarray}
where $V_{i-1}=V_{i-1}^{(1)}\cup V_{i-1}^{(2)}\cup V_{i-1}^{(3)}$.
Moreover, $$\min (V_{i-1}+C_{i-1})=-2c_{i-1}^{(s_1)}+c_{i-1}^{(s_1)}+s_j>0,$$
\begin{eqnarray*}&&\max(W_{i-1}+\{c_{i}^{(1)}, \ldots, c_{i}^{(s_1)}\})=-2 c_{i-1}^{(s_1)}+c_i^{(1)}\\
&<&-2 c_{i-1}^{(s_1)}+s_j+c_i^{(s_1)}
=\min(V_{i-1}^{(1)}+\{c_{i}^{(1)}, \ldots, c_{i}^{(s_1)}\}).\end{eqnarray*}

Combining with the construction in Step 1, write
$$w^{(i-1)}=-2c_{i-1}^{(s_1)}+s_1s_j, \quad i=2,3,\ldots$$

Thus,
\begin{eqnarray}\label{eq5202}&&V_{i-1}^{(1)}+\left\{c_{i}^{(1)}, \ldots, c_{i}^{(s_1)}\right\}\\
&=&\bigcup_{\lambda=1}^{s_1-1}\left[w^{(i-1)}-\lambda s_j+c_{i}^{(s_1)},w^{(i-1)}-\lambda s_j+c_{i}^{(1)}\right]\cup \left[w^{(i-1)}+c_{i}^{(s_1)},d_i+s_1\right].\nonumber
\end{eqnarray}

By (\ref{eq4.5}), we have
\begin{eqnarray*}w^{(i-1)}+c_{i}^{(s_1)}&=&c_{i}^{(s_1)}-2c_{i-1}^{(s_1)}+s_1s_j\\
&=&c_{i}^{(1)}-(s_1-1)-2(c_{i-1}^{(1)}-(s_1-1))+s_1s_j\\
&\leq &c_{i-1}^{(s_1)}=\min (W_{i-1}+C_{i-1}).
\end{eqnarray*}
Combining with (\ref{eq5201}) and (\ref{eq5202}), we know that $$d_{i+1}=w^{(i-1)}+c_{i}^{(s_1)}-r,$$
thus$$d_{i+1}-s_j\frac{\delta-\delta_r}{r}-(\delta_r-1)\notin W_{i}+C_{i}, \text{ if } \delta_r\in \left\{0,1\right\}.$$
$$d_{i+1}-s_j\frac{\delta-\delta_r+r}{r}+(r+1-\delta_r)\notin W_{i}+C_{i}, \text{ if } \delta_r\in \left\{2,\ldots, r-1\right\}.$$

Now, we prove that the complement $C$ is minimal. For $\delta=1,\ldots,s_1$, we divide into two cases:

{\bf Case 1.} $ \delta_r\in \left\{0,1\right\}$. For each $i=1,2,\ldots,$ write
\begin{equation}\label{eq2.7}d_{i}-s_j\frac{\delta-\delta_r}{r}-(\delta_r-1)=c+w, \quad c\in C, w\in W,\end{equation}
 we shall prove that $c=c_i^{(\delta)}$.

First, we shall show that (\ref{eq2.7}) is true for $i=1$.
We have $$d_{1}-s_j\frac{\delta-\delta_r}{r}-(\delta_r-1)-c_1^{(\delta)}=s_1s-ks_1\frac{\delta-\delta_r}{r}\in W_{1}$$
and
$$d_{1}-s_j\frac{\delta-\delta_r}{r}-(\delta_r-1)-c_{1}^{(\nu)}\notin W_{1}, \quad \nu\neq \delta.$$
By (iv-d), (iv-f-1) and (iv-g-1), for all $i\geqslant 2$ and any $\delta\in\{1,\ldots,s_1\}$, we know that $$d_{1}-c_{i}^{(\delta)}-s_j\frac{\delta-\delta_r}{r}-(\delta_r-1)\notin V_{i-1}^{(2)}\cup V_{i-1}^{(3)}.$$

Second, we shall show that (\ref{eq2.7}) is true for $i\geqslant 2$.

By (iv-b), for $i\geq 2$, we have
\begin{equation}\label{eq2.8}d_{i}-s_j\frac{\delta-\delta_r}{r}-(\delta_r-1)-c_{i}^{(\delta)}=d_{i}-c_{i}^{(1)}-ks_1\frac{\delta-\delta_r}{r}\in W.\end{equation}
Moreover, by (iv-e-1), we have
\begin{equation}\label{eq2.9}d_{i}-s_j\frac{\delta-\delta_r}{r}-(\delta_r-1)-c_{i}^{(\nu)}\notin W, \quad \nu\neq \delta.\end{equation}
By (iv-f-1), we have
\begin{equation}\label{eq2.10}d_{i}-c_{i+1}^{(\delta)}-s_j\frac{\delta-\delta_r}{r}-(\delta_r-1)\notin V_{i}^{(2)}\cup V_{i}^{(3)},\end{equation}
where $\delta\in\{1,\ldots,s_1\}$.

By (iv-c) and (iv-g-1) for all $i\geqslant 2$ and $t\geqslant i+2$, we have
\begin{equation}\label{eq2.11}d_{i}-c_{t}^{(\delta)}-s_j\frac{\delta-\delta_r}{r}-(\delta_r-1)\notin V_{t-1}^{(2)}\cup V_{t-1}^{(3)},\end{equation}
where $\delta\in\{1,\ldots,s_1\}$.

By the construction of $V_{i-1}^{(1)}$ in Step 1, we have
$$\min (W\backslash W_{i-1})\geqslant -2c_{i-1}^{(s_1)}+s_j=-2(c_{i-1}^{(1)}-s_1+1)+s_j=-2c_{i-1}^{(1)}+(2+k)s_1-2+r,$$
thus for all $t\leqslant i-1$, we have
\begin{eqnarray*}
&&d_{i}-c_{t}^{(1)}-s_j\frac{\delta-\delta_r}{r}-(\delta_r-1)\\
&&\leq d_{i}-c_{i-1}^{(1)}-s_j\frac{\delta-\delta_r}{r}-(\delta_r-1)\\
&&< d_{i}-c_{i-1}^{(2)}-s_j\frac{\delta-\delta_r}{r}-(\delta_r-1)\\
&&< \cdots \\
&&< d_{i}-c_{i-1}^{(s_1)}-s_j\frac{\delta-\delta_r}{r}-(\delta_r-1)\\
&&<-2c_{i-1}^{(1)},
\end{eqnarray*}
it follows that for all positive integers $t\leqslant i-1$ and any $\delta\in\{1,\ldots,s_1\}$, we have $$d_{i}-c_{t}^{(\delta)}-s_j\frac{\delta-\delta_r}{r}-(\delta_r-1)\notin W\backslash W_{i-1}.$$
Moreover, by (\ref{iv-h-1}), we have
$$d_{i}-s_j\frac{\delta-\delta_r}{r}-(\delta_r-1)\notin W_{i-1}+C_{i-1}.$$
Hence, \begin{equation}\label{eq2.12}d_{i}-s_j\frac{\delta-\delta_r}{r}-(\delta_r-1)\notin W+C_{i-1}.\end{equation}
By(\ref{eq2.8})-(\ref{eq2.12}), we know that (\ref{eq2.7}) is true for $i\geqslant 2$.

{\bf Case 2.} $ \delta_r\in \left\{2,\dots,r-1\right\}$. For each $i=1,2,\ldots,$
write
\begin{equation}\label{eq2.13}d_{i}-s_j\frac{\delta-\delta_r+r}{r}+(r+1-\delta_r)=c+w, \quad c\in C, w\in W.\end{equation}
We shall prove that $c=c_i^{(\delta)}$.

First, we shall show that (\ref{eq2.13}) is true for $i=1$.
We have $$d_{1}-s_j\frac{\delta-\delta_r+r}{r}+(r+1-\delta_r)-c_1^{(\delta)}=s_1s-ks_1\frac{\delta-\delta_r+r}{r}\in W_{1}$$
and
$$d_{1}-s_j\frac{\delta-\delta_r+r}{r}+(r+1-\delta_r)-c_{1}^{(\nu)}\notin W_{1}, \quad \nu\neq \delta.$$

By (iv-f-2) and (iv-g-2), for all $i\geqslant 2$ and any $\delta\in\{1,\ldots,s_1\}$, we know that $$d_{1}-c_{i}^{(\delta)}-s_j\frac{\delta-\delta_r+r}{r}+(r+1-\delta_r)\notin V_{i-1}^{(2)}.$$

Second, we shall show that (\ref{eq2.13}) is true for $i\geqslant 2$.

By (iv-b) for  $i\geqslant 2$ we have
\begin{equation}\label{eq2.14}d_{i}-s_j\frac{\delta-\delta_r+r}{r}+(r+1-\delta_r)-c_{i}^{(\delta)}=d_{i}-c_{i}^{(1)}-ks_1\frac{\delta-\delta_r+r}{r}\in W. \end{equation}
Moreover, by (iv-e-2), we have
\begin{equation}\label{eq2.15}d_{i}-s_j\frac{\delta-\delta_r+r}{r}+(r+1-\delta_r)-c_{i}^{(\nu)}\notin W, \quad \nu\neq \delta. \end{equation}
By (iv-f-2), we have
\begin{equation}\label{eq2.16}d_{i}-c_{i+1}^{(\delta)}-s_j\frac{\delta-\delta_r+r}{r}+(r+1-\delta_r)\notin V_{i,0}^{(2)},\end{equation}
where $\delta\in\{1,\ldots,s_1\}.$

By (iv-g-2), for all $i\geqslant 2$ and $t\geqslant i+2$, we have
\begin{equation}\label{eq2.16}d_{i}-c_{t}^{(\delta)}-s_j\frac{\delta-\delta_r+r}{r}+(r+1-\delta_r)\notin V_{t-1}^{(2)},\end{equation}
where $\delta\in\{1,\ldots,s_1\}.$

For all $t\leqslant i-1$, we have
\begin{eqnarray*}
&&d_{i}-c_{t}^{(1)}-s_j\frac{\delta-\delta_r+r}{r}+(r+1-\delta_r)\\
&&\leq d_{i}-c_{i-1}^{(1)}-s_j\frac{\delta-\delta_r+r}{r}+(r+1-\delta_r)\\
&&< d_{i}-c_{i-1}^{(2)}-s_j\frac{\delta-\delta_r+r}{r}+(r+1-\delta_r)\\
&&< \cdots \\
&&< d_{i}-c_{i-1}^{(s_1)}-s_1\frac{\delta-\delta_r+r}{r}+(r+1-\delta_r)\\
&&<-2c_{i-1}^{(1)},
\end{eqnarray*}
it follows that for all positive integers $t\leqslant i-1$ and any $\delta\in\{1,\ldots,s_1\}$, we have $$d_{i}-c_{t}^{(\delta)}-s_j\frac{\delta-\delta_r+r}{r}+(r+1-\delta_r)\notin W\backslash W_{i-1}.$$
Moreover, by (\ref{iv-h-2}), we have $$d_{i}-s_j\frac{\delta-\delta_r+r}{r}+(r+1-\delta_r)\notin W_{i-1}+C_{i-1}.$$
Hence,
\begin{equation}\label{eq2.17}d_{i}-s_j\frac{\delta-\delta_r+r}{r}+(r+1-\delta_r)\notin W+C_{i-1}.\end{equation}
By (\ref{eq2.14})-(\ref{eq2.17}), we know that (\ref{eq2.13}) is true for $i\geqslant 2$.

Therefore, $C$ is a minimal complement to $W$.

This completes the proof for the case $s_2\geqslant 2$.

\end{document}